\documentclass{article}

\usepackage[utf8]{inputenc}
\usepackage{amsthm,mathtools,thmtools,color,graphicx,hyperref,amssymb,enumerate,cancel,url,authblk, enumitem}

\DeclareMathOperator{\tr}{tr}
\newcommand{\RN}[1]{\MakeUppercase{\romannumeral #1}}

\newtheorem{theorem}{Theorem}[section]
\newtheorem{lemma}[theorem]{Lemma}
\newtheorem{prop}[theorem]{Proposition}

\theoremstyle{definition}
\newtheorem{definition}[theorem]{Definition}

\theoremstyle{remark}
\newtheorem{remark}[theorem]{Remark}

\numberwithin{equation}{section}

\title{GENERALIZED CURVATURE FOR THE OPTIMAL TRANSPORT PROBLEM INDUCED BY A TONELLI LAGRANGIAN}
\date{}
\author{YUCHUAN YANG}
\affil{DEPARTMENT OF MATHEMATICS, UCLA, CALIFORNIA, USA \\ 
\textit{Email address: }\tt yangyuchuan@ucla.edu}

\begin{document}

\maketitle

\begin{abstract}
We propose a generalized curvature that is motivated by the optimal transport problem on $\mathbb{R}^d$ with cost induced by a Tonelli Lagrangian $L$. We show that non-negativity of the generalized curvature implies displacement convexity of the generalized entropy functional on the $L-$Wasserstein space along $C^2$ displacement interpolants.
\end{abstract}


\section{Introduction}
Given a Riemannian manifold $(M,g)$, one may consider the optimal transport problem with cost given by squared Riemannian distance. This induces the \textit{2-Wasserstein distance} $W_2$ on $\mathcal{P}_2(M)$, the space of probability measures on $M$ with finite second moments (i.e. probability measures $\mu$ such that 
\begin{align*}
    \int_M d(x,x_0)^2 d\mu(x)<\infty
\end{align*}
for every $x_0\in M$). The metric space $(\mathcal{P}_2(M),W_2)$ is called the \textit{2-Wasserstein space}, and is known to be a geodesic space (\cite{oldandnew} Chapter 7).

In \cite{OTTO}, Otto proposed that $\mathcal{P}_2 (M)$ admits a formal Riemannian structure and developed a formal calculus on $\mathcal{P}_2 (M)$. This later became what is known as \textit{Otto calculus} \cite{oldandnew} and was made rigorous by Ambrosio-Gigli-Savar\'{e} \cite{AGS}. In particular, Otto calculus allows one to compute \textit{displacement Hessians} of functionals along geodesics in $\mathcal{P}_2 (M)$. This is useful for characterizing a displacement convex functional (i.e. convex along every geodesic) by the non-negativity of its displacement Hessian. In a seminal work by Otto and Villani \cite{OTTOVILLANI}, it was shown that the displacement convexity of the entropy functional is related to the Ricci curvature of $(M,g)$. Since then, the notion of displacement convexity has been useful in many other areas. For instance, it has inspired new heuristics and proofs of various functional inequalities \cite{agueh},\cite{Cordero-Erausquin03inequalitiesfor}. 

Further advances have been made towards understanding the relationships between the geometry of the underlying space and the induced geometry of $\mathcal{P}(M)$, the space of probability measures on $M$.
In his Ph.D. thesis \cite{Schachter2017AnEA}, Schachter studied the optimal transport problem on $\mathbb{R}^d$ with cost induced by a \textit{Tonelli Lagrangian}. The case $d=1$ was considered in \cite{benjamin}, and this work was later used in \cite{hydrodynamics} and \cite{mfg}.

In his work, Schachter developed an \textit{Eulerian calculus}, extending the Otto calculus. Among the other contributions of his thesis, Schachter derived a canonical form for the displacement Hessians of functionals. Using Eulerian calculus, he found a new class of displacement convex functionals on $S^1$ \cite{benjamin}, which includes those found by Carrillo and Slep\v{c}ev in \cite{carrilloslepcev}. In the case when the cost is given by squared Riemannian distance, Schachter proved that his displacement Hessian agrees with Villani's displacement Hessian in \cite{oldandnew}, which is a quadratic form involving the \textit{Bakry-Emery tensor}. 

\textbf{Summary of main results:} In this manuscript, a generalized notion of curvature $\mathcal{K}_x$ (Definition \ref{generalizedcurvature}) is proposed for the manifold $M=\mathbb{R}^d$ equipped with a general Tonelli Lagrangian $L$, and is given by
\begin{align*}
    \mathcal{K}_x (\xi) \coloneqq \tr \bigg(\nabla \xi(x)^2 + A(x,\xi(x))\nabla \xi(x) + B(x,\xi(x))\bigg)
\end{align*}
for vector fields $\xi\in C^2(\mathbb{R}^d;\mathbb{R}^d)$. The maps $A$ and $B$ are defined in Lemma \ref{jacobi}. We prove that this generalized curvature is independent of the choice of coordinates (Theorem \ref{coordinatefree}). In the case where $\xi$ take a special form (that naturally arises from the optimal transport problem), we provide an explicit formula for $\mathcal{K}_x$ in Theorem \ref{formula}. Lastly, we furnish an example of a Lagrangian cost with non-negative generalized curvature that is not given by squared Riemannian distance. This induces a geometry on the $L-$Wasserstein space where the generalized entropy functional \eqref{entropy} is displacement convex along suitable curves.

This paper is organized as follows: In the first four sections, we will review the optimal transport problem induced by a Tonelli Lagrangian, up to and including the notion of displacement convexity. The thesis of Schachter \cite{Schachter2017AnEA} provides a good overview of key definitions and results needed.
Section 2 covers some basic notation. Section 3 reviews some ideas from \cite{Schachter2017AnEA}; chief among them is the relationship between the various formulations of the optimal transport problem. Section 4 discusses functionals along curves in Wasserstein space, including a computation of the displacement Hessian. Section 5 introduces the definition and various properties of the generalized curvature $\mathcal{K}_x$. Lastly, Section 6 provides an example of a Lagrangian with everywhere non-negative generalized curvature.

\subsection*{Acknowledgements} This paper grew out of an undergraduate project I worked on at UCLA. I would like to thank Wilfrid Gangbo (who first introduced this subject to me) for his continued guidance and support. His expertise and generosity have been tremendously helpful. I also thank Tommaso Pacini for helpful feedback. Lastly, I would like to thank the reviewer for the suggestions to the manuscript.

\section{Notation}
We will take our underlying manifold to be $M = \mathbb{R}^d$ and identify its tangent bundle $T\mathbb{R}^d\cong\mathbb{R}^d \times\mathbb{R}^d$. Let $\mathcal{P}^{ac} = \mathcal{P}^{ac}(\mathbb{R}^d)$ denote the set of probability measures on $\mathbb{R}^d$ that are absolutely continuous with respect to the $d-$dimensional Lebesgue measure (denoted $\mathcal{L}^d$). An element of $\mathcal{P}^{ac}$ will often be identified by its density $\rho$. Given $\rho\in\mathcal{P}^{ac}$ and a measurable function $T:\mathbb{R}^d \to \mathbb{R}^d$, $T_{\#}\rho$ will denote the push-forward measure of $\rho$. 

\begin{definition}[Tonelli Lagrangian]
	A function $L:\mathbb{R}^d \times\mathbb{R}^d \to \mathbb{R}$ is called a \textit{Tonelli Lagrangian} if it satisfies the following conditions:
	\begin{enumerate}[label=(\roman*)]
		\item $L\in C^2 (\mathbb{R}^d \times\mathbb{R}^d)$.
		\item For every $x\in\mathbb{R}^d$, the function $L(x,\cdot):\mathbb{R}^d \to\mathbb{R}$ is strictly convex.
		\item $L$ has asymptotic superlinear growth in the variable $v$, in the sense that there exists a constant $c_0\in\mathbb{R}$ and a function $\theta:\mathbb{R}^d \to [0,+\infty)$ with
		\begin{align*}
			\lim_{|v|\to +\infty}\frac{\theta (v)}{|v|}=+\infty
		\end{align*}
		such that
		\begin{align}\label{superlinear}
			L(x,v)\geq c_0 + \theta (v)
		\end{align}
		for all $(x,v)\in\mathbb{R}^d \times\mathbb{R}^d$.
	\end{enumerate}
\end{definition}

Throughout this manuscript, $L \in C^k (\mathbb{R}^d \times\mathbb{R}^d)$, $k\geq 3$ will be assumed to be a Tonelli Lagrangian and we will work with the underlying space $(\mathbb{R}^d , L)$. We denote the gradient with respect to the $x$ (position) and $v$ (velocity) variables by $\nabla_x L, \nabla_v L\in\mathbb{R}^d$ respectively. Similarly, the second-order derivatives will be denoted by $\nabla_{xx}^2 L$, $\nabla_{vv}^2 L$, $\nabla_{xv}^2 L$, $\nabla_{vx}^2L = \nabla_{xv}^2 L^{\top} \in\mathbb{R}^{d\times d}$. We will assume that the Hessian $\nabla_{vv}^2 L(x,v)$ is positive-definite for every $(x,v)\in\mathbb{R}^d \times\mathbb{R}^d$. The time derivative of a function $f(t)$ will be denoted by $\dot{f} = \frac{df}{dt}$.

\section{Optimal transport problem induced by a Tonelli Lagrangian}
\subsection{Lagrangian optimal transport problem}\label{lagrangesubsection}
The goal of this section is to establish the different formulations of the optimal transport problem with cost induced by a Tonelli Lagrangian $L$. In this first subsection, the Lagrangian optimal transport problem will be presented. We will also briefly recall the classical Monge-Kantorovich theory. Most of the material in the subsection can be found in \cite{bernardbuffoni},  \cite{fatfig}, \cite{Schachter2017AnEA}  and \cite{oldandnew}. In subsection \ref{euleriansubsection} we will present an Eulerian perspective and its connections to viscosity solutions of the \textit{Hamilton-Jacobi equation}.

\begin{definition}[Action functional]
	Let $T>0$ and $\gamma\in W^{1,1}([0,T];\mathbb{R}^d)$ be a curve. The \textit{action of $L$ on $\gamma$} is
	\begin{align}\label{action}
		\mathcal{A}_{L,T}(\gamma) = \int_{0}^{T}L(\gamma(t),\dot{\gamma}(t))\;dt.
	\end{align}
	This induces a \textit{cost function} $c_{L,T}: \mathbb{R}^d \times\mathbb{R}^d \to\mathbb{R}$ given by 
	\begin{align}\label{costxy}
		c_{L,T}(x,y) = \inf\{ \mathcal{A}_{L,T}(\gamma)\; :\; \gamma\in W^{1,1}([0,T];\mathbb{R}^d), \gamma(0) = x, \gamma(T) = y\}.
	\end{align} 
	A curve $\gamma$ with $\gamma(0) = x, \gamma(T) = y$ is called an \textit{action-minimizing curve from $x$ to $y$} if $\mathcal{A}_{L,T}(\gamma) = c_{L,T}(x,y)$. 
\end{definition}

\begin{theorem}[\cite{fatfig} Appendix B]\label{actionmincurves}
	For any $x,y\in\mathbb{R}^d$, there exists an action-minimizing curve $\gamma$ from $x$ to $y$ such that
	\begin{enumerate}[label=(\roman*)]
		\item $\mathcal{A}_{L,T}(\gamma) = c_{L,T}(x,y)$
		\item $\gamma\in C^k ([0,T] ;\mathbb{R}^d)$
		\item $\gamma$ satisfies the Euler-Lagrange equation
		\begin{align}\label{eulerlagrange}
			\frac{d}{dt}((\nabla_v L)(\gamma,\dot{\gamma})) = (\nabla_x L)(\gamma,\dot{\gamma})
		\end{align}
	\end{enumerate}
\end{theorem}

\begin{definition}[Lagrangian flow]\label{lagrangianflow}
	The Lagrangian flow $\Phi:[0,+\infty)\times\mathbb{R}^d\times\mathbb{R}^d \to \mathbb{R}^d\times\mathbb{R}^d$ is defined by
	\begin{align*}
		\begin{cases}
			\frac{d}{dt}((\nabla_v L)(\Phi)) = (\nabla_x L)(\Phi)\\
			\Phi(0,x,v)=(x,v)
		\end{cases}
	\end{align*}
\end{definition}

We refer the reader to \cite{fatfig} and \cite{Schachter2017AnEA} for further properties of the cost function $c_{L,T}$. In particular, it is locally Lipschitz and thus differentiable almost everywhere by Rademacher's theorem. Moreover, if either $\frac{\partial}{\partial x}c_{L,T}(x_0 ,y_0)$ or $\frac{\partial}{\partial y}c_{L,T}(x_0 ,y_0)$ exists at $(x_0,y_0)$, then the action-minimizing curve from $x_0$ to $y_0$ is unique. With the cost $c_{L,T}$, we may state the Monge problem and the Kantorovich problem.

\begin{definition}[Monge problem]
	Let $\rho_0,\rho_T\in\mathcal{P}^{ac}$. The Monge optimal transport problem from $\rho_0$ to $\rho_T$ for the cost $c_{L,T}$ is the minimization problem
	\begin{align}
		\inf_M \bigg\{\int_{\mathbb{R}^d}c_{L,T}(x,M(x))\rho_0(x) \;dx\;:\;M_{\#}\rho_0=\rho_T\;,\; M\text{ Borel measurable} \bigg\}.
	\end{align}
\end{definition}

\begin{definition}[Kantorovich problem]\label{kantorovichproblem}
	Let $\Pi (\rho_0,\rho_T)$ denote the set of all probability measures on $\mathbb{R}^d \times\mathbb{R}^d$ with marginals $\rho_0$ and $\rho_T$. Then the Kantorovich optimal transport problem from $\rho_0$ to $\rho_T$ for the cost $c_{L,T}$ is the minimization problem
	\begin{align}\label{kantorovichcost}
		\inf_\pi \bigg\{ \int_{\mathbb{R}^d \times\mathbb{R}^d}c_{L,T}(x,y)\;d\pi(x,y) \;:\;\pi\in\Pi (\rho_0,\rho_T) \bigg\}.
	\end{align}
	A minimizer $\pi$ is called an optimal transport plan. The infimum in \eqref{kantorovichcost} is denoted $W_{c_{L,T}}(\rho_0 , \rho_T)$ and it is called the \textit{Kantorovich cost from $\rho_0$ to $\rho_T$}.
\end{definition}

If $W_{c_{L,T}}(\rho_0 , \rho_T)$ is finite, then the Monge problem with cost $c_{L,T}$ admits an optimizer $M$ (called the Monge map) that is unique $\rho_0-$almost everywhere \cite{fatfig}. Note that the Monge problem is only concerned with the initial and final states (i.e. $\rho_0,\rho_T$). To interpolate between $\rho_0$ and $\rho_T$ in a way that respects the cost $c_{L,T}$, we consider the Lagrangian formulation of the optimal transport problem induced by $L$.

\begin{definition}[Lagrangian optimal transport problem]\label{lagrangian}
	Let $\rho_0,\rho_T\in\mathcal{P}^{ac}$. The Lagrangian optimal transport problem from $\rho_0$ to $\rho_T$ induced by the Tonelli Lagrangian $L$ is the minimization problem
	\begin{align}
		\inf_\sigma \bigg\{ \int_{0}^T \int_{\mathbb{R}^d}L(\sigma(t,x),\dot{\sigma}(t,x))\rho_0(x)\;dx\;dt \bigg\}
	\end{align}
	where the infimum is taken over all $\sigma:[0,T]\times\mathbb{R}^d \to\mathbb{R}^d$ such that
	\begin{enumerate}[label=(\roman*)]
		\item $\sigma(\cdot,x)\in W^{1,1}([0,T];\mathbb{R}^d)$ for every $x\in\mathbb{R}^d$
		\item $\sigma(t,\cdot)$ is Borel measurable for every $t\in [0,T]$
		\item $\sigma(0,x) = x$ for every $x\in\mathbb{R}^d$
		\item $\sigma(T,\cdot)_{\#}\rho_0 = \rho_T$
	\end{enumerate}
\end{definition}

In \cite{Schachter2017AnEA}, it is shown that if $W_{c_{L,T}}(\rho_0 , \rho_T)$ is finite , then the Lagrangian optimal transport problem admits an optimizer $\sigma$ such that $\sigma(\cdot,x)$ is an action-minimizing curve from $\sigma(0,x)=x$ to $\sigma(T,x)$ for every $x\in\mathbb{R}^d$. Moreover, the map $\sigma(T,\cdot)$ coincides with the Monge map $M$ and so is unique $\rho_0-$almost everywhere. With an optimizer $\sigma$, we can define the notion of \textit{displacement interpolation}, which is the analogue of a geodesic in $\mathcal{P}^{ac}$.

\begin{definition}[Displacement interpolant]
	Let $\rho_0,\rho_T\in\mathcal{P}^{ac}$ be such that the Kantorovich cost $W_{c_{L,T}}(\rho_0 , \rho_T)$ is finite. Let $\sigma$ be an optimizer of the Lagrangian optimal transport problem. Then the \textit{displacement interpolant} between $\rho_0$ and $\rho_T$ for the cost $c_{L,T}$ is the measure-valued map
	\begin{align*}
		[0,T]\ni t \mapsto \mu_t = \sigma(t,\cdot)_{\#}\rho_0.
	\end{align*}
\end{definition}

Since $\mu_t$ is absolutely continuous with respect to $\mathcal{L}^d$ for every $t\in [0,T]$ (\cite{fatfig} Theorem 5.1), we will also identify $\mu_t$ with its density $\rho_t$. Subsequently, we will always denote a displacement interpolant by a function $\rho: [0,T]\times\mathbb{R}^d \to \mathbb{R}$ and use the notation $\rho_t = \rho(t,\cdot)$ whenever the intention is clear. Since the maps $\sigma(t,\cdot)$ are uniquely defined ($\rho_0-$almost everywhere) on the support of $\rho_0$, the displacement interpolant is well-defined. Moreover, the map $\sigma\big|_{[0,t]\times\mathbb{R}^d}$ for an intermediary time $t\in[0,T]$ optimizes the Lagrangian optimal transport problem from $\rho_0$ to $\rho_t$, i.e.
\begin{align*}
	W_{c_{L,t}}(\rho_0,\rho_t) = \int_{0}^t \int_{\mathbb{R}^d}L(\sigma(s,x),\dot{\sigma}(s,x))\rho_0(x)\;dx\;ds.
\end{align*}

In order to discuss the Eulerian formulation of the optimal transport problem, we need to introduce the \textit{Kantorovich duality}. We do so in accordance with the convention of \cite{oldandnew}.

\begin{theorem}[Kantorovich duality]
	The Kantorovich optimal transport problem from $\rho_0$ to $\rho_T$ for the cost $c_{L,T}$ has a dual formulation
	\begin{align*}
		&\inf_\pi \bigg\{ \int_{\mathbb{R}^d \times\mathbb{R}^d}c_{L,T}(x,y)\;d\pi(x,y) \;:\;\pi\in\Pi (\rho_0,\rho_T) \bigg\}\\
		= &\sup_{(u_0,u_T)}\bigg\{ \int_{\mathbb{R}^d}u_T(y)\rho_T(y)\;dy  - \int_{\mathbb{R}^d}u_0(x)\rho_0(x)\;dx \;:\;(u_0,u_T)\in L^1 (\rho_0)\times L^1(\rho_T)\;,\\
		&\qquad\qquad u_T(y)-u_0(x)\leq c_{L,T}(x,y)\quad\forall (x,y)\in\mathbb{R}^d \times\mathbb{R}^d\bigg\}
	\end{align*}
	Moreover, we may assume that 
	\begin{align*}
		u_T(y) &= \inf_{x\in\mathbb{R}^d} \big\{ u_0(x) + c_{L,T}(x,y) \big\}\\
		u_0(x) &= \sup_{y\in\mathbb{R}^d} \big\{ u_T(y) - c_{L,T}(x,y) \big\}
	\end{align*}
	If $(u_0, u_T)$ is an optimizer of the dual problem, then $u_0$ and $u_T$ are called Kantorovich potentials.
\end{theorem}

\begin{remark}\label{potentials}
	If the Monge optimal transport problem from $\rho_0$ to $\rho_T$ for the cost $c_{L,T}$ admits a minimizer $M$ (unique $\rho_0-$almost everywhere), then any optimal transport plan $\pi\in\Pi (\rho_0,\rho_T)$ is concentrated on the graph of $M$ \cite{fatfig}. Moreover, if $u_0$ and $u_T$ are Kantorovich potentials, then 
	\begin{align*}
		u_T(y) - u_0(x) \leq c_{L,T}(x,y)
	\end{align*}
	for every $(x,y)\in\mathbb{R}^d \times\mathbb{R}^d$ and we have equality 
	\begin{align*}
		u_T(M(x)) - u_0(x) = c_{L,T}(x,M(x))
	\end{align*}
	for $x$ $\rho_0-$almost everywhere (see \cite{oldandnew} Theorem 5.10).
\end{remark}

\subsection{Eulerian formulation}\label{euleriansubsection}

The paper by Benamou and Brenier \cite{Benamou2000ACF} is one of the earliest works establishing the Eulerian formulation and its connection to Hamilton-Jacobi equations. Subsequently, the relationships between the different formulations of the optimal transport problem were further studied (for instance, \cite{bernardbuffoni}).  

In particular, the Eulerian view establishes the displacement interpolant as a solution to the \textit{continuity equation}. First, we state some basic facts about the \textit{Hamiltonian}. 

The Hamiltonian associated with the Tonelli Lagrangian $L$ is defined as the \textit{Legendre transform} of $L$ with respect to the variable $v$, i.e.
\begin{align}
	H(x,p) = \sup_{v\in\mathbb{R}^d}\{\langle p,v \rangle - L(x,v)\}.
\end{align}
Thus, the Hamiltonian $H$ satisfies the \textit{Fenchel-Young inequality}
\begin{align}
	\langle v,p \rangle \leq H(x,p) + L(x,v)
\end{align}
for all $x,v,p\in\mathbb{R}^d$, with equality if and only if 
\begin{align}\label{youngequality}
	p = (\nabla_v L)(x,v).
\end{align}
Moreover, $H\in C^k (\mathbb{R}^d \times \mathbb{R}^d)$ and
\begin{align}\label{inverse}
	(\nabla_v L)(x,(\nabla_p H)(x,r))= (\nabla_p H)(x, (\nabla_v L)(x,r) ) = r.
\end{align}

Let $u_0:\mathbb{R}^d \to [-\infty,+\infty]$ be a function and $T>0$. 
We define the \textit{Lax-Oleinik evolution} $u:[0,T]\times\mathbb{R}^d\to [-\infty,+\infty]$ of $u_0$ by
\begin{align}
	u(t,x) &\coloneqq \inf_{\gamma} \bigg\{ u_0(\gamma(0)) + \int_{0}^{t}L(\gamma(\tau),\dot{\gamma}(\tau))\;d\tau\;:\;\gamma\in W^{1,1}([0,t];\mathbb{R}^d)\;,\;\gamma(t) = x \bigg\} \label{laxoleinik} \\
	&= \inf_{\gamma} \bigg\{ u_0(\gamma(0)) + \mathcal{A}_{L,t}(\gamma)\;:\;\gamma\in W^{1,1}([0,t];\mathbb{R}^d)\;,\;\gamma(t) = x \bigg\} \nonumber \\
	&=\inf_{y\in\mathbb{R}^d} \big\{ u_0(y) + c_{L,t}(y,x) \nonumber \big\}
\end{align}
so that $u(0,x) = u_0(x)$. 

\begin{remark}\label{finite}
	Since $L$ is bounded below, if there exists some $(t^*,x^*)\in (0,T]\times\mathbb{R}^d$ such that $u(t^*,x^*)$ is finite, then $u$ is finite on all of $[0,T]\times\mathbb{R}^d$.
\end{remark}
It is known that if $u$ is finite, then it is a viscosity solution of the Hamilton-Jacobi equation
\begin{align}\label{hje}
	\frac{\partial u}{\partial t} + H(x,\nabla u) = 0
\end{align}
(see \cite{fathi} Section 7.2 and \cite{fathinew} Theorem 1.1).

\begin{definition}[Calibrated curve]
	Let $f:[t_0,t_1]\times\mathbb{R}^d$ be a function. A curve $\gamma\in W^{1,1}([t_0,t_1];\mathbb{R}^d)$ is called a $(f,L)-$\textit{calibrated curve} if $f(t_0,\gamma(t_0))$, $f(t_1,\gamma(t_1))$ and $\int_{t_0}^{t_1}L(\gamma(t),\dot{\gamma}(t))\;dt$ are all finite and
	\begin{align}
		f(t_1,\gamma(t_1)) - f(t_0,\gamma(t_0)) = \int_{t_0}^{t_1}L(\gamma(t),\dot{\gamma}(t))\;dt.
	\end{align}
\end{definition}

In the following proposition, we mention some properties of $u$ that are of interest to us. The proofs can be found in \cite{cannarsa}, \cite{fathi} and \cite{fathinew}.

\begin{prop}\label{laxoleinikprop}
	Let $u$ be defined as in \eqref{laxoleinik}. If $u$ is finite, then the following hold:
	\begin{enumerate}[label=(\roman*)]
		\item $u$ is continuous and locally semi-concave on $(0,T)\times\mathbb{R}^d$.
		\item $u$ is a viscosity solution of the Hamilton-Jacobi equation
		\begin{align*}
			\frac{\partial u}{\partial t} + H(x,\nabla u) = 0.
		\end{align*}
		\item If $[a,b]\subset [0,T]$ and $\gamma:[a,b]\to\mathbb{R}^d$ is a $(u,L)-$calibrated curve, then $u$ is differentiable at $(t,\gamma(t))$ for every $t\in [a,b]$ and we have
		\begin{align}
			\nabla u(t,\gamma(t)) = (\nabla_v L)(\gamma(t),\dot{\gamma}(t)).
		\end{align}
		\item If $u$ is differentiable at $(t^*, x^*)$, then there is at most one $(u,L)-$calibrated curve $\gamma:[a,b]\to\mathbb{R}^d$ with $t^* \in [a,b]$ and $\gamma(t^*) = x^*$.
	\end{enumerate}
\end{prop}

We now return to the optimal transport problem from $\rho_0\in\mathcal{P}^{ac}$ to $\rho_T\in\mathcal{P}^{ac}$ induced by $L$. Suppose that $W_{c_{L,T}}(\rho_0,\rho_T)$ is finite and let $u_0\in L^1(\rho_0)$ be a Kantorovich potential.

\begin{prop}
	Let $\sigma : [0,T]\times\mathbb{R}^d \to \mathbb{R}^d$ be an optimizer of the Lagrangian optimal transport problem from $\rho_0$ to $\rho_T$ induced by $L$. Let $(u_0,u_T)$ be the corresponding Kantorovich potentials and $u:[0,T]\times\mathbb{R}^d \to \mathbb{R}$ be the Lax-Oleinik evolution of $u_0$. Then $(\nabla u)(t,\sigma(t,x))$ exists for all $t\in [0,T]$ and $x$ $\rho_0-$almost everywhere. In addition, $\sigma$ satisfies the relation
	\begin{align}\label{characteristic}
		\dot{\sigma}(t,x) = (\nabla_p H)(\sigma(t,x), (\nabla u)(t,\sigma(t,x))).
	\end{align}
\end{prop}

\begin{proof}
	By Remark \ref{finite}, $u$ is finite since $u(T,\cdot) = u_T \in L^1(\rho_T)$. By Remark \ref{potentials}, the Kantorovich potentials $(u_0,u_T)$ satisfy 
	\begin{align*}
		u_T(\sigma(T,x)) - u_0(x) &= c_{L,T}(x,\sigma(T,x))\\
		\iff u(T,\sigma(T,x)) - u(0,\sigma(0,x)) &= c_{L,T}(x,\sigma(T,x))
	\end{align*}
	for $x$ $\rho_0-$almost everywhere (recall that $\sigma(T,\cdot)$ coincides with the Monge map). Thus, for $\rho_0-$almost every $x$, the curve $t\mapsto\sigma(t,x)$ is a $(u,L)-$calibrated curve and so $(\nabla u)(t,\sigma(t,x)) = (\nabla_v L)(\sigma(t,x),\dot{\sigma}(t,x))$ exists by Proposition \ref{laxoleinikprop}. Using identity \eqref{inverse}, we get
	\begin{align*}
		\dot{\sigma}(t,x) = (\nabla_p H)(\sigma(t,x), (\nabla u)(t,\sigma(t,x))).
	\end{align*}
\end{proof}

\begin{remark}\label{velocity}
	Let $V:[0,T]\times\mathbb{R}^d \to \mathbb{R}^d$ be a time-dependent vector field that agrees with $(\nabla_p H)(x,(\nabla u)(t,x))$ on the set 
	\begin{align*}
		S_t \coloneqq \{ \sigma(t,y)\in\mathbb{R}^d\;:\; y\in\text{supp}(\rho_0)\;,\;(\nabla u)(t,\sigma(t,y))\;\text{exists} \}
	\end{align*}
	for each $t\in[0,T]$. Using the definition of the displacement interpolant $\rho_t = \sigma(t,\cdot)_{\#}\rho_0$, and the fact that $(\nabla u)(t,\sigma(t,x))$ exists for all $t\in [0,T]$ and $\rho_0-$almost every $x\in\mathbb{R}^d$, we have that the set
	\begin{align*}
		\{ \sigma(t,y)\in\mathbb{R}^d\;:\; y\in\text{supp}(\rho_0)\;,\; u\;\text{not differentiable at }(t,\sigma(t,y))\}
	\end{align*}
	is a set of zero $\rho_t-$measure. Thus, $S_t$ has full $\rho_t-$measure and so $V(t,x) = (\nabla_p H)(x,(\nabla u)(t,x))$ $\rho_t-$almost everywhere. By \eqref{characteristic}, $\dot{\sigma}(t,x) = V(t,\sigma(t,x))$ for all $t\in [0,T]$ and $\rho_0-$almost every $x\in\mathbb{R}^d$. This means that $\rho_t$ and $V$ solve the continuity equation
	\begin{align}
		\frac{\partial \rho_t}{\partial t} + \nabla\cdot (\rho_t V) = 0
	\end{align}
	in the sense of distributions (\cite{Schachter2017AnEA} Proposition 3.4.3).
\end{remark}

\section{Generalized entropy functional and displacement Hessian}
Otto calculus and Schachter's Eulerian calculus both allow for explicit computations, assuming that all relevant quantities possess sufficient regularity. However, the regularity of a displacement interpolant $\rho$ depends on the Lagrangian $L$, the initial and final densities $(\rho_0, \rho_T)$, and the optimal trajectories $\sigma$ (or the velocity field $V$ in the Eulerian framework). In general, the Kantorovich potential $u_0$ arising from an optimal transport problem induced by a Tonelli Lagrangian $L$ is only known to be semiconcave, differentiable $\mathcal{L}^d -$almost everywhere, and its gradient $\nabla u_0$ is only locally bounded (see \cite{figgig} and \cite{geo} Appendix C). This implies that the initial velocity $V(0,x) = (\nabla_p H)(x,\nabla u_0 (x))$ is only locally bounded. As such, even if the initial density $\rho_0$ is smooth, its regularity may fail to propagate along the displacement interpolant.

For our purpose of computing displacement Hessians, we require displacement interpolants to be of class $C^2$. Fortunately, such displacement interpolants do exist and we can construct them if we impose two additional criteria on $L$.

\subsection{$C^2$ displacement interpolants}

Let $L\in C^{k+1} (\mathbb{R}^d \times\mathbb{R}^d)$, $k\geq 3$ be a Tonelli Lagrangian satisfying two additional criteria (see \cite{cannarsa} Chapters 6.3, 6.4).
\begin{enumerate}[label=(L\arabic*)]
	\item There exists $\tilde{c}_0\geq 0$ and $\tilde{\theta}:[0,\infty)\to [0,\infty)$ with 
	\begin{align*}
		\lim_{r\to +\infty}\frac{\tilde{\theta}(r)}{r} = +\infty
	\end{align*}
	such that
	\begin{align*}
		L(x,v)\geq \tilde{\theta}(|v|)-\tilde{c}_0.
	\end{align*}
	In addition, $\tilde{\theta}$ is such that for any $M>0$ there exists $K_M >0$ with
	\begin{align*}
		\tilde{\theta}(r+m)\leq K_M[1+\tilde{\theta}(r)]
	\end{align*}
	for all $m\in [0,M]$ and all $r\geq 0$.
	\item For any $r>0$, there exists $C_r > 0$ such that
	\begin{align*}
		|(\nabla_x L)(x,v)| + |(\nabla_v L)(x,v)| < C_r \tilde{\theta}(|v|)
	\end{align*}
	for all $|x|\leq r$, $v\in\mathbb{R}^d$.
\end{enumerate}

Some common examples of Tonelli Lagrangians satisfying these criteria include the \textit{Riemannian kinetic energy}
\begin{align*}
    L(x,v) = \frac{1}{2}g_x(v,v)
\end{align*}
where $g_x$ denotes the underlying Riemannian metric tensor, and Lagrangians that arise from mechanics
\begin{align*}
    L(x,v) = \frac{1}{2}g_x(v,v) + U(x)
\end{align*}
for some appropriate potential $U:\mathbb{R}^d\to\mathbb{R}$.

Let $H$ be the corresponding Hamiltonian.

\begin{lemma}\label{smoothflow}
	Let $u_0\in C^{k+1}(\mathbb{R}^d)$ with $u_0(x)\geq -\tilde{c}_0$ for all $x\in\mathbb{R}^d$. Let $u:[0,+\infty)\times\mathbb{R}^d \to [-\infty,+\infty]$ be the Lax-Oleinik evolution of $u_0$, as defined in \eqref{laxoleinik}. For $x\in\mathbb{R}^d$, consider the Lagrangian flow (introduced in Definition \ref{lagrangianflow})
	\begin{align*}
		\Phi(t,x,V(0,x)) = (\Phi_1(t,x,V(0,x)), \Phi_2(t,x,V(0,x)))\;,\quad t\in [0,+\infty)
	\end{align*}
	where $V:[0,+\infty)\times\mathbb{R}^d \to \mathbb{R}^d$ is a time-dependent vector field defined by 
	\begin{align*}
		V(t,x) = (\nabla_p H)(x,(\nabla u)(t,x)).
	\end{align*}
    (Here, $\Phi_1$ and $\Phi_2$ are the $x$ and $v$ components of $\Phi$ respectively.)
	If we let $\sigma(t,x) = \Phi_1(t,x,V(0,x))$, then $\dot{\sigma}(t,x) = V(t,\sigma(t,x))$ for all $t\in [0,+\infty)$, $x\in\mathbb{R}^d$. Moreover, $\sigma(t,\cdot):\mathbb{R}^d \to \mathbb{R}^d$ is a $C^{k}-$diffeomorphism for every $t\in [0,+\infty)$.
\end{lemma}

\begin{proof}
	Since $L$ and $u_0$ are both bounded below, we have
	\begin{align*}
		u(t,x) &= \inf_{\gamma} \bigg\{ u_0(\gamma(0)) + \int_{0}^{t}L(\gamma(\tau),\dot{\gamma}(\tau))\;d\tau\;,\;\gamma(t)=x\bigg\}\\
		&\geq  -\tilde{c}_0  - \tilde{c}_0 t\\
		&> -\infty
	\end{align*}
	and so $u$ is finite. From \cite{fathinew}, $u$ is a continuous viscosity solution of the Hamilton-Jacobi equation \eqref{hje} and we know that for each $(t,x)\in (0,+\infty)\times\mathbb{R}^d$, there exists a unique $(u,L)-$calibrated curve $\gamma_x : [0,t]\to\mathbb{R}^d$ such that $\gamma_x(t) = x$. Moreover, $(\nabla u)(s,\gamma_x (s))$ exists for all $s\in [0,t]$ and is given by 
	\begin{align*}
		(\nabla u)(s,\gamma_x (s)) &= (\nabla_v L)(\gamma_x (s),\dot{\gamma}_x (s)) \\
		\iff \dot{\gamma}_x (s) &= (\nabla_p H)(\gamma_x (s) , (\nabla u) (s,\gamma_x (s)))\\
		\iff \dot{\gamma}_x (s) &= V(s,\gamma_x (s))
	\end{align*}
	Since each $\gamma_x$ is necessarily an action-minimizing curve from $\gamma_x (0)$ to $\gamma_x (t) = x$, it is the unique solution curve to the Euler-Lagrange system
	\begin{align*}
		\begin{cases}
			\frac{d}{dt}((\nabla_v L)(\gamma,\dot{\gamma})) = (\nabla_x L)(\gamma,\dot{\gamma})\\
			\gamma (0) = \gamma_x (0)\\
			\dot{\gamma}(0) = \dot{\gamma}_x (0)
		\end{cases}
	\end{align*}
	Therefore, $\sigma(t,\cdot):\mathbb{R}^d \to \mathbb{R}^d$ is a bijection for all $t\in [0,+\infty)$ and $\dot{\sigma}(t,x) = V(t,\sigma(t,x))$ for all $(t,x)\in [0,+\infty)\times\mathbb{R}^d$. 
	
	Lastly, since $L\in C^{k+1} (\mathbb{R}^d \times\mathbb{R}^d)$ and $u_0\in C^{k+1} (\mathbb{R}^d)$, we have that $u\in C^{k+1} ([0,+\infty)\times\mathbb{R}^d)$ \cite{cannarsa}. As $\nabla_p H \in C^{k}(\mathbb{R}^d \times\mathbb{R}^d ; \mathbb{R}^d)$, we have $V(t,\cdot)\in C^{k}(\mathbb{R}^d ; \mathbb{R}^d)$ and so $\sigma(t,\cdot):\mathbb{R}^d \to \mathbb{R}^d$ is a $C^{k}-$diffeomorphism for every $t\in [0,+\infty)$.
\end{proof}

\begin{prop}
	Let $\rho_0\in\mathcal{P}^{ac} \cap C_{c}^2 (\mathbb{R}^d)$ be a compactly supported density. Then for any $T>0$, there exists a $C^2$ displacement interpolant $\rho: [0,T]\times\mathbb{R}^d \to \mathbb{R}$ with $\rho(0,\cdot) = \rho_0$.
\end{prop}

\begin{proof}
	Let $u_0, u, V, \sigma$ be defined as in Lemma \ref{smoothflow} and fix $T>0$. For $t\in [0,T]$, define
	\begin{align*}
		\rho(t,\cdot) = \sigma(t,\cdot)_{\#}\rho_0.
	\end{align*}
	We claim that $\sigma$ is an optimizer of the Lagrangian optimal transport problem from $\rho_0$ to $\rho_T = \rho(T,\cdot)$, which would imply that $\rho$ is indeed a displacement interpolant. Let $\phi:[0,T]\times\mathbb{R}^d \to \mathbb{R}^d$ satisfy the four conditions in Definition \ref{lagrangian}. By Lemma \ref{smoothflow}, $t\mapsto \sigma(t,x)$ is a $(u,L)-$calibrated curve for each $x\in\mathbb{R}^d$. Thus, for every $x\in\mathbb{R}^d$,
	\begin{align*}
		u(T,\sigma(T,x)) - u(0,\sigma(0,x)) &= \int_{0}^{T}L(\sigma(t,x),\dot{\sigma}(t,x))\;dt \\
		\iff u(T,\sigma(T,x)) - u_0(x) &= \int_{0}^{T}L(\sigma(t,x),\dot{\sigma}(t,x))\;dt \\
		\implies \int_{\mathbb{R}^d} [u(T,\sigma(T,x)) - u_0(x)]\rho_0(x)\;dx &= \int_{\mathbb{R}^d}\int_{0}^{T}L(\sigma(t,x),\dot{\sigma}(t,x))\rho_0(x)\;dt\;dx 
	\end{align*}
	By the definition of pushforward measure, the LHS of the last equality is
	\begin{align*}
		\int_{\mathbb{R}^d} [u(T,\sigma(T,x)) - u_0(x)]\rho_0(x)\;dx &= \int_{\mathbb{R}^d} u(T,y)\rho_T(y)\;dy - \int_{\mathbb{R}^d} u_0(x)\rho_0(x)\;dx\\
		&= \int_{\mathbb{R}^d} [u(T,\phi(T,x)) - u(0,\phi(0,x))]\rho_0(x)\;dx
	\end{align*}
	where the last equality is due to the assumption that $\phi(T,\cdot)_{\#}\rho_0 = \rho_T$ and $\phi(0,x) = x$. By the definition of $u$ (i.e. \eqref{laxoleinik}), we have that 
	\begin{align*}
		u(T,\phi(T,x)) - u(0,\phi(0,x)) &\leq \int_{0}^{T}L(\phi(t,x),\dot{\phi}(t,x))\;dt\\
		\implies \int_{\mathbb{R}^d} [u(T,\phi(T,x)) - u(0,\phi(0,x))]\rho_0(x)\;dx &\leq \int_{\mathbb{R}^d} \int_{0}^{T}L(\phi(t,x),\dot{\phi}(t,x))\rho_0(x)\;dt\;dx
	\end{align*}
	for every $x\in\mathbb{R}^d$. Thus, 
	\begin{align*}
		\int_{\mathbb{R}^d}\int_{0}^{T}L(\sigma(t,x),\dot{\sigma}(t,x))\rho_0(x)\;dt\;dx  \leq \int_{\mathbb{R}^d} \int_{0}^{T}L(\phi(t,x),\dot{\phi}(t,x))\rho_0(x)\;dt\;dx.
	\end{align*}
	Since $\phi$ was arbitrary, $\sigma$ is indeed an optimizer of the Lagrangian optimal transport problem from $\rho_0$ to $\rho_T$.
	
	By Lemma \ref{smoothflow}, $\sigma(t,\cdot):\mathbb{R}^d \to \mathbb{R}^d$ is a $C^{k}-$diffeomorphism for every $t\in [0,T]$ and $\sigma(\cdot,x)\in C^{k+1}([0,T];\mathbb{R}^d)$ for every $x\in\mathbb{R}^d$. Using the change-of-variables formula,
	\begin{align*}
		\rho(t,y) &= \frac{\rho_0(\cdot)}{|\text{det} \nabla \sigma(t,\cdot)|}\bigg|_{[\sigma(t,\cdot)]^{-1}(y)}\\
		&= \frac{\rho_0(\cdot)}{\text{det} \nabla \sigma(t,\cdot)}\bigg|_{[\sigma(t,\cdot)]^{-1}(y)}
	\end{align*}
	where $\text{det} \nabla \sigma(t,\cdot) >0$ because $\sigma(0,x) = x \implies \text{det} \nabla \sigma(0,\cdot) = 1$. Since $k\geq 3$, $\rho\in C^2([0,T]\times\mathbb{R}^d)$.

\end{proof}

\subsection{Displacement Hessian}
Let $F\in C^2 ((0,+\infty))\cap C([0,+\infty))$ be a function satisfying
\begin{enumerate}[label=(F\arabic*)]
	\item $F(0) = 0$ \label{F1},
	\item $s^2 F''(s)\geq sF'(s) - F(s) \geq 0,$ $\quad\forall s\in [0,+\infty)$.\label{F2}
\end{enumerate}

If $\rho_0 \in \mathcal{P}^{ac}$ is such that $F(\rho_0)\in L^1 (\mathbb{R}^d)$, we define the \textit{generalized entropy functional}
\begin{align}
	\mathcal{F}(\rho_0) = \int_{\mathbb{R}^d}F(\rho_0 (x))\;dx. \label{entropy}
\end{align}
This is well-defined at least on $\mathcal{P}^{ac}\cap C_{c}^0 (\mathbb{R}^d)$ since $F(0)= 0$ implies
\begin{align*}
	\int_{\mathbb{R}^d}F(\rho_0 (x))\;dx = \int_{\text{supp}(\rho_0)}F(\rho_0(x))\;dx 
\end{align*}
which is finite.

\begin{remark}
	If $\rho_0$ is the density of a fluid and $\mathcal{F}(\rho_0)$ is the internal energy, then $\rho_0 F'(\rho_0) - F(\rho_0)$ can be interpreted as a pressure \cite{FGY},\cite{oldandnew}.
\end{remark}

\begin{definition}[Displacement convexity]
	The generalized entropy functional $\mathcal{F}$ is said to be convex along a displacement interpolant $\rho_t$, $t\in [0,T]$, if $\mathcal{F}(\rho_t)$ is finite and
	\begin{align}
		\mathcal{F}(\rho_t) \leq \frac{T-t}{T}\mathcal{F}(\rho_0) + \frac{t}{T}\mathcal{F}(\rho_T)
	\end{align}
	for every $t\in [0,T]$. $\mathcal{F}$ is said to be \textit{displacement convex} if it is convex along every displacement interpolant (on which $\mathcal{F}$ is real-valued).
\end{definition}

\begin{remark}
	When the displacement interpolant is a ``straight line", McCann proved that $\mathcal{F}$ is displacement convex if $s\mapsto s^d F(s^{-d})$ is convex and non-increasing on $(0,+\infty)$ \cite{MCCANN1997153}. In this context, a ``straight line" displacement interpolant refers to one of the form 
	\begin{align*}
		\rho_t = \bigg( \frac{T-t}{T}\text{id} + \frac{t}{T}M  \bigg)_{\#}\rho_0.
	\end{align*}
	where $M$ is the Monge map between $\rho_0$ and $\rho_T$.
\end{remark}

Along a suitable displacement interpolant $\rho_t$, if the map $t\mapsto \mathcal{F}(\rho_t)$ is $C^2$, then the condition that $\frac{d^2}{dt^2}\mathcal{F}(\rho_t)\geq 0$ ensures convexity of $\mathcal{F}$ along $\rho_t$. The following displacement Hessian formula is a special case of Theorem 4.3.2 of \cite{Schachter2017AnEA}.

\begin{theorem}[Displacement Hessian formula]
	Let $\rho\in C^2 ([0,T]\times\mathbb{R}^d)$ be a displacement interpolant, with $\rho_0 = \rho(0,\cdot)$ compactly supported. Let $\sigma:[0,T]\times\mathbb{R}^d \to \mathbb{R}^d$ be an optimizer of the Lagrangian optimal transport problem from $\rho_0$ to $\rho_T$. Let $V:[0,T]\times\mathbb{R}^d \to \mathbb{R}^d$ be defined as in Remark \ref{velocity} so that $\rho,V$ satisfy the continuity equation $\dot{\rho} = -\nabla\cdot (\rho V)$. Assume that $\sigma$ and $V$ are $C^2$ at least on the set
	\begin{align*}
		\bigcup_{t\in [0,T]} \{t\}\times\textnormal{supp}(\rho_t).
	\end{align*}
	Then $\frac{d^2}{dt^2}\mathcal{F}(\rho)$ exists for every $t\in [0,T]$ and is given by
	\begin{align}\label{displacementhessian}
		\frac{d^2}{dt^2}\mathcal{F}(\rho) = \int_{\mathbb{R}^d}(\rho G'(\rho)-G(\rho))(\nabla\cdot V)^2 + G(\rho)(\tr((\nabla V)^2)-\nabla\cdot W)\;dx
	\end{align}
	where $G:[0,+\infty)\to\mathbb{R}$ is defined by
	\begin{align}
		G(s) &= sF'(s) - F(s)\\
		G'(s) &= sF''(s)
	\end{align}
	and
	\begin{align}
		W = \dot{V} + \nabla VV.
	\end{align}
\end{theorem}

\begin{remark}
	The requirement that $\rho_0$ is compactly supported serves to ensure that $\mathcal{F}$ is finite along $\rho$. In addition, the compactness of supp$(\rho_0)$ and the continuity of $\sigma$ together ensures that the set
	$\{\sigma(t,x)\;:\; t\in [0,T]\;,\; x\in\text{supp}(\rho_0)\}$ is compact. Thus,
	\begin{align*}
		\Sigma \coloneqq \bigcup_{t\in [0,T]}\text{supp}(\rho_t)
	\end{align*}
	is bounded, up to a set of zero $\mathcal{L}^d -$measure. This means that $\frac{d^2}{dt^2}\mathcal{F}(\rho)$ exists for every $t\in [0,T]$ and satisfies
	\begin{align*}
		\frac{d^2}{dt^2}\mathcal{F}(\rho) &= \frac{d^2}{dt^2} \int_{\Sigma}F(\rho(t,x))\;dx\\
		&= \int_{\Sigma}\frac{d^2}{dt^2}F(\rho(t,x))\;dx.
	\end{align*}
\end{remark}

\begin{remark}
	By Remark \ref{velocity}, for every $t\in[0,T]$, $V(t,\cdot)$ is uniquely defined on supp$(\rho_t)$ $\rho_t-$almost everywhere. Thus, \eqref{displacementhessian} is well-defined.
\end{remark}

\begin{proof}
	The displacement Hessian is
	\begin{align*}
		\frac{d^2}{dt^2}\mathcal{F}(\rho) &= \int F''(\rho)\dot{\rho}^2 + F'(\rho)\ddot{\rho}\;dx\\
		&= \int F''(\rho)\dot{\rho}^2 - F'(\rho)\nabla\cdot(\dot{\rho}V + \rho\dot{V})\;dx
	\end{align*}
	Integrating by parts, the above expression becomes
	\begin{align*}
		\int F''(\rho)\dot{\rho}^2 + \langle \nabla (F'(\rho)) , \dot{\rho}V + \rho\dot{V}\rangle\;dx\\
		= \int F''(\rho)\bigg(\dot{\rho}^2 + \langle\nabla\rho, \dot{\rho}V+\rho\dot{V}\rangle\bigg)\;dx.
	\end{align*}
	Using the continuity equation $\dot{\rho} = -\nabla\cdot (\rho V)$, the definitions of $W$ and $G$, and integration by parts, this integral can be written as
	\begin{align*}
		\int \rho G'(\rho)(\nabla\cdot V)^2 - G(\rho)\nabla\cdot\bigg((\nabla\cdot V)V-\nabla VV + W\bigg)\;dx.
	\end{align*}
	A straightforward computation then yields the desired formula.
\end{proof}

\begin{remark}
	Recall that $\rho G'(\rho) - G(\rho) = \rho^2 F''(\rho) - \rho F'(\rho) + F'(\rho)\geq 0$ and $G(\rho) = \rho F'(\rho) - F(\rho) \geq 0$ by assumption \ref{F2}. Thus, the condition that $\tr((\nabla V)^2)-\nabla\cdot W \geq 0$ would ensure that $\frac{d^2}{dt^2}\mathcal{F}(\rho)\geq 0$. In the case where the cost is given by squared Riemannian distance, the term $\tr((\nabla V)^2)-\nabla\cdot W$ is a quadratic form involving the Bakry-Emery tensor \cite{Schachter2017AnEA}, \cite{oldandnew}. In the following section, we will generalize this quadratic form for an arbitrary Tonelli Lagrangian.
\end{remark}

\section{Generalized curvature for Tonelli Lagrangians}
The goal of this section is to define a generalized curvature for the space $(\mathbb{R}^d, L)$. In principle, this generalized curvature is similar to the Ricci curvature in the sense that it is related to the deformation of a shape flowing along action-minimizing curves. The generalized curvature, however, will not be a tensor because it will depend on both the tangent vector and its gradient. Throughout this section, we will assume that $L$ is a $C^3$ Tonelli Lagrangian.

Let $T>0$ and $\sigma:[0,T]\times\mathbb{R}^d \to \mathbb{R}^d$ be such that  
\begin{align*}
	\begin{cases}
		\frac{d}{dt}\big( (\nabla_v L)(\sigma(t,x),\dot{\sigma}(t,x)) \big) = (\nabla_x L)(\sigma(t,x),\dot{\sigma}(t,x))\;, \quad\forall (t,x)\in [0,T]\times\mathbb{R}^d\\
		\sigma(0,x) = x\;, \quad\forall x\in\mathbb{R}^d\\
		\sigma(t,\cdot) : \mathbb{R}^d\to \mathbb{R}^d \text{ is a } C^3-\text{diffeomorphism for every } t\in [0,T]
	\end{cases}
\end{align*}
Let $V:[0,T]\times\mathbb{R}^d \to \mathbb{R}^d$ be a time-dependent vector field defined by $\dot{\sigma}(t,x) = V(t,\sigma(t,x))$ so that $V(t,\cdot)\in C^2(\mathbb{R}^d ; \mathbb{R}^d)$ for every $t\in [0,T]$.
Following the method outlined in \cite{oldandnew} Chapter 14, we first derive Lemma \ref{jacobi}, which is an ODE of the Jacobian matrix $\nabla\sigma$.

\begin{lemma}\label{jacobi}
	Define $A,B:\mathbb{R}^d \times\mathbb{R}^d \to \mathbb{R}^{d\times d}$ by
	\begin{align}
		A(x,v) &= (\nabla_{vv}^2 L)(x,v)^{-1}\bigg[ \frac{d}{dt}\big((\nabla_{vv}^2 L)(\gamma_{x,v}(t),\dot{\gamma}_{x,v}(t))\big)\bigg|_{t=0} + (\nabla_{vx}^2 L)(x,v) - (\nabla_{xv}^2 L)(x,v) \bigg]\\
		B(x,v) &= (\nabla_{vv}^2 L)(x,v)^{-1}\bigg[  \frac{d}{dt}\big((\nabla_{vx}^2 L)(\gamma_{x,v}(t),\dot{\gamma}_{x,v}(t))\big)\bigg|_{t=0} - (\nabla_{xx}^2 L)(x,v) \bigg]
	\end{align}
	where $\gamma_{x,v}:[0,\epsilon) \to\mathbb{R}^d$ is the unique curve satisfying the Euler-Lagrange equation with initial conditions $\gamma_{x,v}(0) = x$, $\dot{\gamma}_{x,v}(0) = v$. Then the Jacobian $\nabla\sigma$ satisfies a second-order matrix equation
	\begin{align}
		\nabla\ddot{\sigma} + A(\sigma,\dot{\sigma})\nabla\dot{\sigma} + B(\sigma,\dot{\sigma})\nabla\sigma = 0.
	\end{align}
\end{lemma}

\begin{proof}
	Taking the spatial gradient of the Euler-Lagrange equation, 
	\begin{align*}
		0 &= \nabla_x \bigg( \frac{d}{dt}\big( (\nabla_v L)(\sigma,\dot{\sigma}) \big) - (\nabla_x L)(\sigma,\dot{\sigma}) \bigg) \\
		&= \frac{d}{dt}\bigg((\nabla_{vx}^2 L)(\sigma,\dot{\sigma})\nabla\sigma + (\nabla_{vv}^2 L)(\sigma,\dot{\sigma})\nabla\dot{\sigma}\bigg) - (\nabla_{xx}^2 L)(\sigma,\dot{\sigma})\nabla\sigma - (\nabla_{xv}^2 L)(\sigma,\dot{\sigma})\nabla\dot{\sigma}\\
		&= \frac{d}{dt}\big((\nabla_{vx}^2 L)(\sigma,\dot{\sigma})\big)\nabla\sigma + (\nabla_{vx}^2 L)(\sigma,\dot{\sigma})\nabla\dot{\sigma} + \frac{d}{dt}\big((\nabla_{vv}^2 L)(\sigma,\dot{\sigma})\big)\nabla\dot{\sigma} + (\nabla_{vv}^2 L)(\sigma,\dot{\sigma})\nabla\ddot{\sigma}\\
		&\quad- (\nabla_{xx}^2 L)(\sigma,\dot{\sigma})\nabla\sigma - (\nabla_{xv}^2 L)(\sigma,\dot{\sigma})\nabla\dot{\sigma}
	\end{align*}
	To conclude, we group by the terms $\nabla\sigma, \nabla\dot{\sigma}, \nabla\ddot{\sigma}$ and multiply by $(\nabla_{vv}^2 L)(\sigma,\dot{\sigma})^{-1}$.
\end{proof}

\begin{lemma}
	Define
	\begin{align}
		\mathcal{U}(t,x) =
		(\nabla V)(t,\sigma(t,x)).
	\end{align}
	Then
	\begin{align}\label{trace}
		\dot{\mathcal{U}} + \mathcal{U}^2 + A(\sigma,\dot{\sigma})\mathcal{U} + B(\sigma,\dot{\sigma}) = 0.
	\end{align}
\end{lemma}

\begin{proof}
	First, we note that since $\dot{\sigma}(t,x) = V(t,\sigma(t,x))$, we get
	\begin{align*}
		(\nabla\dot{\sigma})(t,x) &= (\nabla V)(t,\sigma(t,x))(\nabla\sigma)(t,x)\\
		\implies (\nabla V)(t,\sigma(t,x)) &= (\nabla\dot{\sigma})(t,x) ((\nabla\sigma)(t,x))^{-1}
	\end{align*}
	and so 
	\begin{align*}
		\mathcal{U}(t,x) = (\nabla\dot{\sigma})(t,x) ((\nabla\sigma)(t,x))^{-1}.
	\end{align*}
	
	Using the matrix identity $\frac{d}{dt}M^{-1} = -M^{-1}\dot{M}M^{-1}$,
	\begin{align*}
		\dot{\mathcal{U}} &= (\nabla\ddot{\sigma})(\nabla\sigma)^{-1} - (\nabla\dot{\sigma})(\nabla\sigma)^{-1}(\nabla\dot{\sigma})(\nabla\sigma)^{-1}\\ 
        &= (\nabla\ddot{\sigma})(\nabla\sigma)^{-1} - \mathcal{U}^2.
	\end{align*}
	By Lemma \ref{jacobi}, $\nabla\ddot{\sigma} = -A(\sigma,\dot{\sigma})(\nabla\dot{\sigma}) - B(\sigma,\dot{\sigma})(\nabla\sigma)$ and so
	\begin{align*}
		0 &= \dot{\mathcal{U}} + \mathcal{U}^2 + \bigg(A(\sigma,\dot{\sigma})(\nabla\dot{\sigma}) - B(\sigma,\dot{\sigma})(\nabla\sigma)\bigg)(\nabla\sigma)^{-1}\\
		&= \dot{\mathcal{U}} + \mathcal{U}^2 + A(\sigma,\dot{\sigma})\mathcal{U} + B(\sigma,\dot{\sigma}).
	\end{align*}
\end{proof}

We want to show that the term $\tr((\nabla V)^2)-\nabla\cdot W$ appearing in the displacement Hessian formula \eqref{displacementhessian} arises from equation \eqref{trace}. Taking the trace of \eqref{trace}, we have
\begin{align}
	\frac{d}{dt}\bigg((\nabla\cdot V)(t,\sigma)\bigg) + \tr\bigg((\nabla V)(t,\sigma)^2 + A(\sigma,\dot{\sigma})(\nabla V)(t,\sigma) + B(\sigma,\dot{\sigma})\bigg) = 0.
\end{align}
On the other hand, direct computation yields
\begin{align*}
	\frac{d}{dt}\bigg((\nabla\cdot V)(t,\sigma)\bigg) &= (\nabla\cdot\dot{V})(t,\sigma) + \langle V(t,\sigma), (\nabla(\nabla\cdot V))(t,\sigma) \rangle.
\end{align*}
Since $V(t,\sigma(t,x)) = \dot{\sigma}(t,x)$ and $\sigma(0,x) = x$, we may restate the above equation as
\begin{align}\label{bochner}
	&(\nabla\cdot\dot{V})(t,x) + \langle V(t,x), (\nabla(\nabla\cdot V))(t,x) \rangle \nonumber \\ 
	&+ \tr\bigg((\nabla V)(t,x)^2 + A(x,V(t,x))(\nabla V)(t,x) + B(x,V(t,x))\bigg) = 0.
\end{align}
Using the identities
\begin{align*}
	\nabla\cdot((\nabla\cdot V)V) = (\nabla\cdot V)^2 + \langle V,\nabla(\nabla\cdot V) \rangle
\end{align*}
and 
\begin{align*}
	\dot{V} = -\nabla VV + W ,
\end{align*}
we see that
\begin{align*}
	\nabla\cdot\dot{V} + \langle V,\nabla(\nabla\cdot V) \rangle &= \nabla\cdot (-\nabla VV + W) + \nabla\cdot((\nabla\cdot V)V) - (\nabla\cdot V)^2.
\end{align*}
By the computation of the displacement Hessian from the previous section, this is precisely $\tr((\nabla V)^2)-\nabla\cdot W$.

At this point, \eqref{bochner} holds for all time-dependent $C^2$ vector fields whose integral curves satisfy the Euler-Lagrange equation \eqref{eulerlagrange}. To show that \eqref{bochner} holds for an arbitrary fixed vector field, we first need to make sense of the term $\dot{V}$ by introducing Definition \ref{gamma}.

\begin{prop}\label{extend}
	Given any fixed vector field $V_0 \in C^2(\mathbb{R}^d ; \mathbb{R}^d)$, we may extend it for a short time to a unique time-dependent vector field $V(t,x)$, $t\in [0,\epsilon)$ with the following properties:
	\begin{enumerate}[label=(\roman*)]
		\item $V(0,\cdot) = V_0$
		\item The integral curves of $V$ satisfy the Euler-Lagrange equation, i.e. 
		\begin{align*}
			\dot{\sigma}(t,x) &= V(t,\sigma(t,x))\\
			\sigma(0,x) &= x\\
			\frac{d}{dt}((\nabla_v L)(\sigma,\dot{\sigma})) &= (\nabla_x L)(\sigma,\dot{\sigma})
		\end{align*}
	\end{enumerate}
\end{prop}

\begin{proof}
	We recall Definition \ref{lagrangianflow} and the existence of a Lagrangian flow $\Phi = (\Phi_1,\Phi_2)$ satisfying $\frac{d}{dt}((\nabla_v L)(\Phi)) = (\nabla_x L)(\Phi)$. Set $\sigma(t,x) = \Phi_1(t,x,V_0(x))$. The maps $\sigma(t,\cdot):\mathbb{R}^d \to \mathbb{R}^d$ are defined for all $t\in[0,+\infty)$ and there exists $\epsilon>0$ such that $\sigma(t,\cdot)$ is invertible for $t\in[0,\epsilon)$.
	Thus, for $t\in[0,\epsilon)$, we may define the desired vector field by
	\begin{align}
		V(t,y) = \dot{\sigma}(t,\sigma^{-1}(t,y)).
	\end{align}
\end{proof}

\begin{definition}\label{gamma}
	Given a Tonelli Lagrangian $L$, we define the operation 
	\begin{align*}
		\Gamma_L : C^2(\mathbb{R}^d ; \mathbb{R}^d) &\to C^2(\mathbb{R}^d ; \mathbb{R}^d)\\
		V_0 &\mapsto \dot{V}(0,\cdot)
	\end{align*}
	as in Proposition \ref{extend}.
\end{definition}

\begin{remark}
	By the Euler-Lagrange equation \eqref{eulerlagrange}, we can give an explicit formula for $\Gamma_L(V_0)$. Suppose $\sigma(t,x)$ and $V(t,x)$ satisfy the two properties in Proposition \ref{extend}, then
	\begin{align*}
		\ddot{\sigma}(t,x) = \dot{V}(t,\sigma(t,x))+(\nabla V)(t,\sigma(t,x))V(t,\sigma(t,x)).
	\end{align*}
	Since 
	\begin{align*}
		\ddot{\sigma}(t,x) = (\nabla_{vv}^2 L)(x,V(t,x))^{-1}\bigg( (\nabla_x L)(x,V(t,x)) - (\nabla_{vx}^2 L)(x,V(t,x))V(t,x) \bigg)
	\end{align*}
	by the Euler-Lagrange equation, we have
	\begin{align*}
		(\Gamma_L(V_0))(x) &= \dot{V}(0,x)\\
		&= (\nabla_{vv}^2 L)(x,V_0(x))^{-1}\bigg( (\nabla_x L)(x,V_0(x)) - (\nabla_{vx}^2 L)(x,V_0(x))V_0(x) \bigg) \\
		&\quad - (\nabla V_0 (x))V_0(x).
	\end{align*}
\end{remark}

\begin{definition}[Generalized curvature]\label{generalizedcurvature}
	Let $\xi\in C^2(\mathbb{R}^d ; \mathbb{R}^d)$. For each $x\in\mathbb{R}^d$, we define the \textit{generalized curvature} $\mathcal{K}_x$ by
	\begin{align}
		\mathcal{K}_x (\xi) \coloneqq \tr \bigg(\nabla \xi(x)^2 + A(x,\xi(x))\nabla \xi(x) + B(x,\xi(x))\bigg)
	\end{align}
	where $A,B:\mathbb{R}^d \times\mathbb{R}^d \to \mathbb{R}^{d\times d}$ are defined as in Lemma \ref{jacobi}.
\end{definition}

\begin{theorem}\label{coordinatefree}
	Let $\xi\in C^2(\mathbb{R}^d ; \mathbb{R}^d)$. Then
	\begin{align}
		-\big(\nabla\cdot\big(\Gamma_L (\xi)\big)\big)(x) - \langle \xi(x),(\nabla(\nabla\cdot \xi))(x) \rangle = \mathcal{K}_x(\xi).
	\end{align}
	In particular, the generalized curvature $\mathcal{K}_x$ is intrinsic, i.e. does not depend on the choice of coordinates.
\end{theorem}

\begin{proof}
	By Proposition \ref{extend}, we may extend $\xi$ for a short time to a time-dependent vector field $V(t,x)$, with $V(0,\cdot) = \xi$, whose integral curves satisfy the Euler-Lagrange equation. Thus, \eqref{bochner} holds for $V$ and we have
	\begin{align*}
		\mathcal{K}_x(\xi) &= \mathcal{K}_x(V(0,\cdot))\\
		&\overset{\eqref{bochner}}{=} -(\nabla\cdot\dot{V})(0,x) - \langle V(0,x),(\nabla(\nabla\cdot V))(0,x) \rangle\\
		&= -\big(\nabla\cdot\big(\Gamma_L (\xi)\big)\big)(x) - \langle \xi(x),(\nabla(\nabla\cdot \xi))(x) \rangle
	\end{align*}
	To show that $\mathcal{K}_x$ is intrinsic, we will show that the operator 
	\begin{align}
		\xi\mapsto -\nabla\cdot\big(\Gamma_L (\xi)\big) - \langle \xi,\nabla(\nabla\cdot \xi) \rangle
	\end{align}
	is invariant under a change of coordinates. By Definition 
	\ref{gamma} and the definition of divergence, $-\nabla\cdot\big(\Gamma_L (\xi)\big)$ is coordinate-free. Next, observe that $\langle \xi,\nabla(\nabla\cdot \xi) \rangle$ is the directional derivative of $\nabla\cdot\xi$ (which is coordinate-free) with respect to $\xi$. Thus,
	\begin{align*}
		\langle \xi(x),\nabla(\nabla\cdot \xi)(x) \rangle &= \lim_{h\to 0}\frac{(\nabla\cdot\xi)(x + h\xi(x)) - (\nabla\cdot\xi)(x)}{h}
	\end{align*}
	is also coordinate-free.
\end{proof}

In the case where $\xi(x) = (\nabla_p H)(x,\nabla u(x)) \iff \nabla u(x) = (\nabla_v L)(x,\xi(x))$ for some potential $u:\mathbb{R}^d \to \mathbb{R}$ (cf. Proposition \ref{laxoleinikprop}, Lemma \ref{smoothflow}), we can derive an explicit formula for $\mathcal{K}_x(\xi)$.

\begin{theorem}[Formula for $\mathcal{K}_x(\xi)$]\label{formula}
	Let $\xi\in C^2(\mathbb{R}^d ; \mathbb{R}^d)$ be such that there exists $u\in C^2(\mathbb{R}^d)$, with
	\begin{align*}
		\nabla u(x) = (\nabla_v L)(x, \xi(x))\;,\quad \forall x\in\mathbb{R}^d.
	\end{align*}
	Then,
	\begin{align*}
		\mathcal{K}_x(\xi) 
		&= L^{ik}\frac{\partial \xi_j}{\partial x_k}\frac{\partial^2 L}{\partial v_j \partial v_l}\frac{\partial \xi_l}{\partial x_i} 
		- L^{im}\frac{\partial^3 L}{\partial v_m \partial v_j \partial v_k}\xi_l \frac{\partial \xi_j}{\partial x_i}\frac{\partial \xi_k}{\partial x_l}\\
		&\quad + L^{im}\frac{\partial^3 L}{\partial v_m \partial v_j \partial v_k}L^{kl}\frac{\partial L}{\partial x_l}\frac{\partial \xi_j}{\partial x_i}
		- L^{ir}\frac{\partial^3 L}{\partial v_r \partial v_j \partial v_k}L^{kl}\frac{\partial^2 L}{\partial x_l \partial v_m}\frac{\partial \xi_j}{\partial x_i}\xi_m\\
		&\quad - L^{kl}\frac{\partial^3 L}{\partial x_k \partial v_j \partial v_l}\frac{\partial \xi_j}{\partial x_i}\xi_i
		+ L^{ij}\frac{\partial^3 L}{\partial x_i \partial v_j \partial v_k}L^{kl}\frac{\partial L}{\partial x_l}\\
		&\quad - L^{ij}\frac{\partial^3 L}{\partial x_i \partial v_j \partial v_k}L^{kl}\frac{\partial^2 L}{\partial x_l \partial v_m}\xi_m - L^{ij}\frac{\partial^2 L}{\partial x_j \partial x_i}
	\end{align*}
	where all terms involving $L$ are evaluated at $(x,\xi(x))$.
\end{theorem}

\begin{proof}
	See Appendix.
\end{proof}

In conclusion, the displacement Hessian formula in \eqref{displacementhessian} can be written as
\begin{align}
	\frac{d^2}{dt^2}\mathcal{F}(\rho) = \int_{\mathbb{R}^d}(\rho G'(\rho)-G(\rho))(\nabla\cdot V)^2 + G(\rho)\mathcal{K}_x (V)\;dx.
\end{align}

\section{Displacement convexity for a non-Riemannian Lagrangian cost}
In this section, we provide an example of a Lagrangian cost that is not a squared Riemannian distance. We prove using a perturbation argument that the corresponding generalized curvature is non-negative and thus the generalized entropy functional is convex along $C^2$ displacement interpolants.

Let $g(x)$ be a positive definite matrix for every $x\in\mathbb{R}^d$ so that $\frac{1}{2}\langle v, g(x)v\rangle$, $v\in\mathbb{R}^d$ defines a Riemannian metric. Let $g_{ij} = g_{ij}(x)$ denote the $ij$-th entry of $g(x)$ and $g^{ij} = g^{ij}(x)$ denote the $ij$-th entry of the inverse matrix $g(x)^{-1}$. Further assume that the $g_{ij}$ are bounded with bounded derivatives, and that the corresponding Bakry-Emery tensor (denoted BE$_g$) is bounded from below. That is,
\begin{align*}
	\text{BE}_g = \text{Ric}+\nabla^2 \bigg(\frac{1}{2}\text{log}(\text{det }g)\bigg)\geq k_g >0.
\end{align*}

Define the Lagrangian
\begin{align*}
	L(x,v) = \frac{1}{2}\langle v, g(x)v\rangle
\end{align*}
and the perturbed Lagrangian
\begin{align*}
	\tilde{L}(x,v) = \frac{1}{2}\langle v, g(x)v\rangle + \varphi(v)
\end{align*}
where $\varphi:\mathbb{R}^d \to \mathbb{R}$ is a smooth perturbation (for instance, take $\varphi$ to be of Schwartz class). Using Theorem \ref{formula}, the respective generalized curvatures are given by
\begin{align*}
	\mathcal{K}_x(\xi) 
	&= \underbrace{g^{ik}\frac{\partial \xi_j}{\partial x_k}g_{jl}\frac{\partial \xi_l}{\partial x_i}}_{\RN{1}} - 
	\underbrace{g^{kl}\frac{\partial g_{jl}}{\partial x_k}\frac{\partial \xi_j}{\partial x_i}\xi_i}_{\RN{5}} \\
	&\quad 
	+ \underbrace{g^{ij}\frac{\partial g_{jk}}{\partial x_i}g^{kl}\frac{\partial g_{mn}}{\partial x_l}\xi_m \xi_n}_{\RN{6}}\\
	&\quad - \underbrace{g^{ij}\frac{\partial g_{jk}}{\partial x_i}g^{kl}\frac{\partial g_{nm}}{\partial x_l}\xi_n\xi_m}_{\RN{7}} 
	- \underbrace{g^{ij}\frac{\partial^2 g_{kl}}{\partial x_j \partial x_i}\xi_k \xi_l}_{\RN{8}}
\end{align*}
and
\begin{align*}
	\tilde{\mathcal{K}}_x(\xi) 
	&= \underbrace{\tilde{L}^{ik}\frac{\partial \xi_j}{\partial x_k}\frac{\partial^2 \tilde{L}}{\partial v_j \partial v_l}\frac{\partial \xi_l}{\partial x_i}}_{\tilde{\RN{1}}} 
	- \underbrace{\tilde{L}^{im}\frac{\partial^3 \varphi}{\partial v_m \partial v_j \partial v_k}\xi_l \frac{\partial \xi_j}{\partial x_i}\frac{\partial \xi_k}{\partial x_l}}_{\tilde{\RN{2}}}\\
	&\quad + \underbrace{\tilde{L}^{im}\frac{\partial^3 \varphi}{\partial v_m \partial v_j \partial v_k}\tilde{L}^{kl}\frac{\partial g_{nr}}{\partial x_l}\xi_n \xi_r \frac{\partial \xi_j}{\partial x_i}}_{\tilde{\RN{3}}}
	- \underbrace{\tilde{L}^{ir}\frac{\partial^3 \varphi}{\partial v_r \partial v_j \partial v_k}\tilde{L}^{kl}\frac{\partial g_{mn}}{\partial x_l}\xi_n
		\frac{\partial \xi_j}{\partial x_i}\xi_m}_{\tilde{\RN{4}}}\\
	&\quad - \underbrace{\tilde{L}^{kl}\frac{\partial g_{jl}}{\partial x_k}
		\frac{\partial \xi_j}{\partial x_i}\xi_i}_{\tilde{\RN{5}}}
	+ \underbrace{\tilde{L}^{ij}\frac{\partial g_{jk}}{\partial x_i}
		\tilde{L}^{kl}\frac{\partial g_{mn}}{\partial x_l}\xi_m \xi_n}_{\tilde{\RN{6}}}\\
	&\quad - \underbrace{\tilde{L}^{ij}\frac{\partial g_{jk}}{\partial x_i}\tilde{L}^{kl}\frac{\partial g_{mn}}{\partial x_l}\xi_n \xi_m}_{\tilde{\RN{7}}} 
	- \underbrace{\tilde{L}^{ij}\frac{\partial^2 g_{kl}}{\partial x_j \partial x_i}\xi_k \xi_l}_{\tilde{\RN{8}}}
\end{align*}

By Theorem A.3.1 of \cite{Schachter2017AnEA} and \eqref{bochner}, $\mathcal{K}_x(\xi) = ||g^{-1}\nabla\xi^\top||_{\text{HS}}^2 + \text{BE}_g(\xi)$, where $||\cdot||_{\text{HS}}$ denotes the Hilbert-Schmidt norm. Thus, we have a lower bound
\begin{align*}
	\mathcal{K}_x(\xi) = ||g^{-1}\nabla\xi^\top||_{\text{HS}}^2 + \text{BE}_g(\xi) \geq c_g||\nabla\xi||^2 + k_g ||\xi||^2
\end{align*}
where $c_g>0$ is a constant depending on $g$. Fix $\epsilon>0$ such that $\epsilon\leq \min\{\frac{c_g}{10},\frac{k_g}{12}\}$. Our goal is to choose $\varphi$ so that 
\begin{enumerate}
	\item $|\tilde{L}^{ij} - L^{ij}| = |\tilde{L}^{ij} - g^{ij}|$ is sufficiently small, i.e. $||\nabla^2 \varphi||$ is close to zero, and
	\item $|\frac{\partial^3 \varphi}{\partial v_i \partial v_j \partial v_k}|$ is sufficiently small.
\end{enumerate}

To this end, we choose $\varphi$ such that
\begin{align*}
	|\tilde{\mathcal{K}}_x(\xi) - \mathcal{K}_x(\xi)|&\leq 
	|\tilde{\RN{1}}-\RN{1}|+
	|\tilde{\RN{5}}-\RN{5}|+
	|\tilde{\RN{6}}-\RN{6}|+
	|\tilde{\RN{7}}-\RN{7}|+
	|\tilde{\RN{8}}-\RN{8}|\\
	&\quad + |\tilde{\RN{2}}|+|\tilde{\RN{3}}|+|\tilde{\RN{4}}|\\
	&\leq \epsilon ||\nabla\xi||^2 + 2\epsilon ||\nabla\xi||||\xi|| + \epsilon ||\xi||^2 + \epsilon ||\xi||^2 + \epsilon ||\xi||^2 \\
	&\quad + \epsilon||\nabla\xi||^2 + 2\epsilon ||\nabla\xi||||\xi|| + 2\epsilon ||\nabla\xi||||\xi||\\
	&\leq  5\epsilon ||\nabla\xi||^2 + 6\epsilon||\xi||^2\\
	&\leq \frac{c_g}{2}||\nabla\xi||^2 +  \frac{k_g}{2}||\xi||^2
\end{align*}

Since $\mathcal{K}_x(\xi) \geq c_g||\nabla\xi||^2 + k_g ||\xi||^2$, we conclude that $\tilde{\mathcal{K}}_x(\xi) \geq 0$.

\section{Appendix}
The generalized curvature $\mathcal{K}_x(\xi)$ is given by 
\begin{align*}
	\mathcal{K}_x (\xi) \coloneqq \tr \bigg(\nabla \xi(x)^2 + A(x,\xi(x))\nabla \xi(x) + B(x,\xi(x))\bigg).
\end{align*}
In the computations below, time derivatives of $\xi$ will be treated by extending $\xi$ for a short time (in the sense of Proposition \ref{extend}).

\begin{lemma}
	\begin{align*}
		\tr \bigg(\nabla \xi(x)^2 + A(x,\xi(x))\nabla \xi(x) \bigg) &= L^{ik}\frac{\partial \xi_j}{\partial x_k}\frac{\partial^2 L}{\partial v_j \partial v_l}\frac{\partial \xi_l}{\partial x_i} 
		- L^{im}\frac{\partial^3 L}{\partial v_m \partial v_j \partial v_k}\xi_l \frac{\partial \xi_j}{\partial x_i}\frac{\partial \xi_k}{\partial x_l}\\
		&\quad + L^{im}\frac{\partial^3 L}{\partial v_m \partial v_j \partial v_k}L^{kl}\frac{\partial L}{\partial x_l}\frac{\partial \xi_j}{\partial x_i}\\
		&\quad - L^{ir}\frac{\partial^3 L}{\partial v_r \partial v_j \partial v_k}L^{kl}\frac{\partial^2 L}{\partial x_l \partial v_m}\frac{\partial \xi_j}{\partial x_i}\xi_m
	\end{align*}
\end{lemma}

\begin{proof}
	Recall that
	\begin{align*}
		A(x,v) &= (\nabla_{vv}^2 L)(x,v)^{-1}\bigg[ \frac{d}{dt}\big((\nabla_{vv}^2 L)(\gamma_{x,v}(t),\dot{\gamma}_{x,v}(t))\big)\bigg|_{t=0} + (\nabla_{vx}^2 L)(x,v) - (\nabla_{xv}^2 L)(x,v) \bigg].
	\end{align*}
	By assumption, there exists a potential $u(x)$ satisfying
	\begin{align*}
		\nabla u(x) = (\nabla_v L)(x,\xi).
	\end{align*}
	Since the Hessian
	\begin{align*}
		\nabla^2 u(x) = (\nabla_{vx}^2L)(x,\xi) + (\nabla_{vv}^2L)(x,\xi)\nabla\xi 
	\end{align*}
	is symmetric, we have
	\begin{align*}
		(\nabla_{vx}^2L)(x,\xi) + (\nabla_{vv}^2L)(x,\xi)\nabla\xi &= (\nabla_{xv}^2L)(x,\xi) + \nabla\xi^\top (\nabla_{vv}^2L)(x,\xi)\\
		\implies (\nabla_{vv}^2 L)(x,\xi)^{-1}\bigg[(\nabla_{vx}^2 L)(x,\xi) - (\nabla_{xv}^2 L)(x,\xi) \bigg] &= (\nabla_{vv}^2 L)(x,\xi)^{-1}\nabla\xi^\top (\nabla_{vv}^2L)(x,\xi) - \nabla\xi
	\end{align*}
	Next, 
	\begin{align*}
		\frac{d}{dt}\big((\nabla_{vv}^2 L)(x,\xi)\big)_{ij} &= \langle \bigg(\nabla_v\frac{\partial^2 L}{\partial v_i \partial v_j}\bigg)(x,\xi) , \dot{\xi} \rangle\\
		&= \langle \bigg(\nabla_v\frac{\partial^2 L}{\partial v_i \partial v_j}\bigg)(x,\xi) , -\nabla\xi\xi + (\nabla_{vv}^2 L)(x,\xi)^{-1}\big[(\nabla_x L)(x,\xi) - (\nabla_{xv}^2 L)(x,\xi)\xi \big] \rangle\\
		&= -\frac{\partial^3 L}{\partial v_i \partial v_j \partial v_k} \frac{\partial \xi_k}{\partial x_l}\xi_l + \frac{\partial^3 L}{\partial v_i \partial v_j \partial v_k} L^{kl}\frac{\partial L}{\partial x_l} - \frac{\partial^3 L}{\partial v_i \partial v_j \partial v_k} L^{kl}\frac{\partial^2 L}{\partial x_l \partial v_m}\xi_m
	\end{align*}
	
\end{proof}

\begin{lemma}
	\begin{align*}
		\tr\bigg(B(x,\xi(x))\bigg) &= - L^{kl}\frac{\partial^3 L}{\partial x_k \partial v_j \partial v_l}\frac{\partial \xi_j}{\partial x_i}\xi_i
		+ L^{ij}\frac{\partial^3 L}{\partial x_i \partial v_j \partial v_k}L^{kl}\frac{\partial L}{\partial x_l}\\
		&\quad - L^{ij}\frac{\partial^3 L}{\partial x_i \partial v_j \partial v_k}L^{kl}\frac{\partial^2 L}{\partial x_l \partial v_m}\xi_m - L^{ij}\frac{\partial^2 L}{\partial x_j \partial x_i}
	\end{align*}
\end{lemma}

\begin{proof}
	Recall that
	\begin{align*}
		B(x,v) &= (\nabla_{vv}^2 L)(x,v)^{-1}\bigg[  \frac{d}{dt}\big((\nabla_{vx}^2 L)(\gamma_{x,v}(t),\dot{\gamma}_{x,v}(t))\big)\bigg|_{t=0} - (\nabla_{xx}^2 L)(x,v) \bigg].
	\end{align*}
	By a similar computation as the previous lemma, we have
	\begin{align*}
		\frac{d}{dt}\big((\nabla_{vx}^2 L)(x,\xi)\big)_{ij} &= \langle \bigg(\nabla_v\frac{\partial^2 L}{\partial v_i \partial x_j}\bigg)(x,\xi) , \dot{\xi} \rangle\\
		&= \langle \bigg(\nabla_v\frac{\partial^2 L}{\partial v_i \partial x_j}\bigg)(x,\xi) , -\nabla\xi\xi + (\nabla_{vv}^2 L)(x,\xi)^{-1}\big[(\nabla_x L)(x,\xi) - (\nabla_{xv}^2 L)(x,\xi)\xi \big] \rangle\\
	\end{align*}
\end{proof}

Putting together these two lemmas, we get the formula for $\mathcal{K}_x(\xi)$.

\bibliographystyle{amsplain}
\bibliography{references.bib}

\providecommand{\bysame}{\leavevmode\hbox to3em{\hrulefill}\thinspace}
\providecommand{\MR}{\relax\ifhmode\unskip\space\fi MR }
\providecommand{\MRhref}[2]{%
  \href{http://www.ams.org/mathscinet-getitem?mr=#1}{#2}
}
\providecommand{\href}[2]{#2}
\begin{thebibliography}{10}

\bibitem{agueh}
M.~Agueh, \emph{{Sharp Gagliardo–Nirenberg Inequalities and Mass Transport
  Theory}}, Journal of Dynamics and Differential Equations \textbf{18} (2006),
  1069--1093.

\bibitem{AGS}
L.~Ambrosio, N.~Gigli, and G.~Savar\'{e}, \emph{{Gradient Flows in Metric
  Spaces and in the Space of Probability Measures}}, Birkh{\"a}user, 2005.

\bibitem{hydrodynamics}
M.~Bauer and K.~Modin, \emph{{Semi-invariant Riemannian metrics in
  hydrodynamics}}, Calculus of Variations and Partial Differential Equations
  \textbf{59} (2020), no.~2, Paper No. 65, 25 pp.

\bibitem{Benamou2000ACF}
J.-D. Benamou and Y.~Brenier, \emph{{A computational fluid mechanics solution
  to the Monge-Kantorovich mass transfer problem}}, Numerische Mathematik
  \textbf{84} (2000), 375--393.

\bibitem{bernardbuffoni}
P.~Bernard and B.~Buffoni, \emph{{Optimal mass transportation and Mather
  theory}}, Journal of the European Mathematical Society \textbf{9} (2007),
  85--121.

\bibitem{cannarsa}
P.~Cannarsa and C.~Sinestrari, \emph{{Semiconcave Functions, Hamilton—Jacobi
  Equations, and Optimal Control}}, Birkh{\"a}user, 2004.

\bibitem{carrilloslepcev}
J.~A. Carrillo and D.~Slep\v{c}ev, \emph{{Example of a displacement convex
  functional of first order}}, Calculus of Variations and Partial Differential
  Equations \textbf{36} (2008), 547--564.

\bibitem{Cordero-Erausquin03inequalitiesfor}
D.~Cordero-Erausquin, W.~Gangbo, and C.~Houdré, \emph{{Inequalities for
  generalized entropy and optimal transportation}}, Contemp. Math. \textbf{353}
  (2004), 73--94.

\bibitem{fathi}
A.~Fathi, \emph{{Weak KAM Theorem in Lagrangian Dynamics}}, Monograph
  \textbf{88} (2014).

\bibitem{fathinew}
\bysame, \emph{{Viscosity solutions of the Hamilton-Jacobi equation on a
  non-compact manifold}},  (2020).

\bibitem{fatfig}
A.~Fathi and A.~Figalli, \emph{{Optimal transportation on non-compact
  manifolds}}, Israel Journal of Mathematics \textbf{175} (2007), 1--59.

\bibitem{FGY}
A.~Figalli, W.~Gangbo, and T.~Yolcu, \emph{{A variational method for a class of
  parabolic {PDEs}}}, Annali della Scuola Normale Superiore di Pisa - Classe di
  Scienze \textbf{Ser. 5, 10} (2011), no.~1, 207--252.

\bibitem{figgig}
A.~Figalli and N.~Gigli, \emph{{Local semiconvexity of {Kantorovich} potentials
  on non-compact manifolds}}, ESAIM: Control, Optimisation and Calculus of
  Variations \textbf{17} (2011), no.~3, 648--653.

\bibitem{geo}
W.~Gangbo and R.~J. McCann, \emph{{The geometry of optimal transportation}},
  Acta Mathematica \textbf{177} (1996), no.~2, 113 -- 161.

\bibitem{mfg}
D.~A. Gomes and T.~Seneci, \emph{{Displacement convexity for first-order
  mean-field games}}, Minimax Theory and its Applications \textbf{3} (2018),
  no.~2, 261–284.

\bibitem{MCCANN1997153}
R.~J. McCann, \emph{{A Convexity Principle for Interacting Gases}}, Advances in
  Mathematics \textbf{128} (1997), no.~1, 153--179.

\bibitem{OTTO}
F.~Otto, \emph{{The Geometry of Dissipative Evolution Equations: The Porous
  Medium Equation}}, Comm Partial Differential Equations \textbf{26} (2001),
  101--174.

\bibitem{OTTOVILLANI}
F.~Otto and C.~Villani, \emph{{Generalization of an Inequality by Talagrand and
  Links with the Logarithmic Sobolev Inequality}}, Journal of Functional
  Analysis \textbf{173} (2000), no.~2, 361--400.

\bibitem{Schachter2017AnEA}
B.~Schachter, \emph{{An Eulerian Approach to Optimal Transport with
  Applications to the Otto Calculus}}, Ph.D. thesis, 2017.

\bibitem{benjamin}
\bysame, \emph{{A New Class of First Order Displacement Convex Functionals}},
  SIAM Journal on Mathematical Analysis \textbf{50} (2018), no.~2, 1779--1789.

\bibitem{oldandnew}
C.~Villani, \emph{{Optimal transport -- Old and new}}, Grundlehren Math. Wiss,
  vol. 338, Springer, Berlin, 2008.

\end{thebibliography}

\end{document}